\documentclass[11pt]{article}
\usepackage{amsmath,amssymb,amsthm}
\usepackage{graphicx}
\usepackage{hyperref}
\usepackage{subfig}

\newtheorem{theorem}{Theorem}
\newtheorem{lemma}[theorem]{Lemma}

\newtheorem{observation}{Observation}

\newtheorem{definition}{Definition}
\newtheorem{conjecture}{Conjecture}

\DeclareMathOperator{\order}{order}

\begin{document}
\title{Shuffling with ordered cards}
\author{Steve Butler\thanks{UCLA, { butler@math.ucla.edu}}~\thanks{Supported by an NSF Postdoctoral fellowship.} \and Ron Graham\thanks{UCSD, { graham@ucsd.edu}}}
\date{\empty}
\maketitle

\begin{abstract}
We consider a problem of shuffling a deck of cards with ordered labels.  Namely we split the deck of $N=k^tq$ cards (where $t\geq 1$ is maximal) into $k$ equally sized stacks and then take the top card off of each stack and sort them by the order of their labels and add them to the shuffled stack.  We show how to find stacks of cards invariant and periodic under the shuffling.  We also show when $\gcd(q,k)=1$  the possible periods of this shuffling are all divisors of $\order_k(N-q)$.
\end{abstract}



\section{Introduction}\label{sec:introduction}
There are many ways to shuffle a deck of cards. One of the most common is to split the deck into equal halves and then riffle or dovetail shuffle the cards back together, wherein the cards from the two halves interlace.  A perfect shuffle of this type is one where the cards alternate perfectly between the two halves.  There are two types of perfect shuffles, depending on which card ends up on top. These are called ``in'' and ``out'' shuffles and have been frequently described in the literature (see, e.g.\ \cite{perfect} or \cite{morris}).

Another way to think about this shuffling is  we split the deck in two equal stacks and go through the stack from top to bottom using a rule about how to put the two cards into the newly shuffled stack.  In and out shuffles correspond to the two rules where either we always put the card from the second half of the stack on top or we always put the card from the first half of the stack on top.

In this paper we start the examination of a new kind of shuffling where our rule is to set an order on the labeling of the cards (we allow for labels to be used multiple times in the deck) and we again go through from top to bottom but now we let the order of the labeling on the cards determine which one goes on top.  That is, starting with $N=kn$ cards with labels from the ordered list $\mathcal{A}_1\succ \mathcal{A}_2\succ\cdots\succ \mathcal{A}_j$, divide the stack of cards into $k$ stacks of $n$, then take the top card off of each stack, sort the $k$ cards (according to the order of the labels, where if the labels agree then the order they are sorted is unimportant) and add them to the new stack.  An example of this is shown in Figure~\ref{fig:sorting} where the labels are $2\succ 1\succ 0$ and we let the heaviest weights ``sink down'' (i.e., go to the lower card height where we start counting from the top card down).

\begin{figure}
\centering
\includegraphics[scale=1.1]{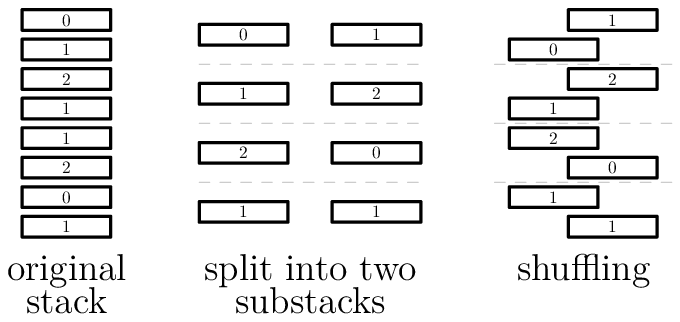}
\caption{An example of shuffling with ordered cards taking the deck $01211201$ to the deck $10212011$.}
\label{fig:sorting}
\end{figure}

Equivalently, this is the same as starting out with a list of $kn$ labeled  objects,
\[
a_0,a_1,\ldots,a_{n-1},a_{n},a_{n+1},\ldots,a_{2n-1},a_{2n},\ldots,a_{kn-1},
\]
and putting this into a $k{\times} n$ matrix where we proceed by filling up the rows left to right and top to bottom,
\[
\begin{pmatrix}
a_0&a_{1}&\cdots&a_{n-1}\\
a_{n}&a_{n+1}&\cdots&a_{2n-1}\\
\vdots&\vdots&\ddots&\vdots\\
\cdot&\cdot&\cdots&a_{kn-1}
\end{pmatrix}.
\]
Now take this matrix and sort the elements in each {\em column}\/ according to the ordering of the labels,
\[
\begin{pmatrix}
b_0&b_{k}&\cdots&\cdot\\
b_1&b_{k+1}&\cdots&\cdot\\
\vdots&\vdots&\ddots&\vdots\\
b_{k-1}&b_{2k-1}&\cdots&b_{kn-1}
\end{pmatrix}
\]
and finally concatenate the columns to form the new list.
\[
b_0,b_1,\ldots,b_{k-1},b_{k},b_{k+1},\ldots,b_{2k-1},b_{2k},\ldots,b_{kn-1}.
\]

As an example if we have $k=3$ and start with $N=12$ cards labeled $021100122110$ (where we again let the 3 labels be ordered $2\succ1\succ0$) then we have
\[
021100122110\longrightarrow
\begin{pmatrix}
0&2&1&1\\
0&0&1&2\\
2&1&1&0
\end{pmatrix}
\longrightarrow
\begin{pmatrix}
2&2&1&2\\
0&1&1&1\\
0&0&1&0
\end{pmatrix}
\longrightarrow
200210111210.
\]
This process can be repeated and so we have
\[
021100122110 \longrightarrow 200210111210 
\longrightarrow 211200110210 
\longrightarrow 200210111210
\]
and now we see we have a stack which will repeat itself every two shuffling steps, or in other words a {\em periodic}\/ stack.  A stack which returns to itself after one shuffling step will similarly be called a {\em fixed}\/ stack.

In general, we can relate the shuffling to a directed graph where each of the possible $j^{N}$ stacks are the vertices and we put a directed edge between two vertices if shuffling one stack gives the other.  Since the outcome of our shuffling is uniquely determined by the order of the cards, each edge will only have one edge going out (though it is possible  many edges can go in).  This immediately gives us the following.

\begin{observation}
Given any stack of cards after applying the shuffling procedure finitely many times we will settle into a stack which is periodic under the shuffling.
\end{observation}

There now arise  several natural questions. For example, how do we find fixed/periodic stacks?  What periods are possible?  How many periodic/fixed stacks are there? How long does it take for a stack to settle into a periodic stack?

In this paper we will answer some of these questions.  In Section~\ref{sec:cycles} we will introduce a weight function on the subscripts and show how to use this to represent the shuffling by a poset structure. In Section~\ref{sec:gcd} we will show in the case when $\gcd(q,k)=1$ (where $N=k^tq$)  the possible periods are all divisors of $\order_k(N-q)$.  In Section~\ref{sec:posets} we will show how to adopt the poset structures to find posets that generate all fixed and periodic stacks.  Finally, in Section~\ref{sec:conclusion} we will give some concluding remarks.

We will throughout assume  the number of cards is $N=k^tq=kn$, where $t\geq 1$ is the highest power of $k$ that divides $N$, and $n=k^{t-1}q$ is the size of the stacks  we split $N$ into when shuffling.  For simplicity we will focus on subscripts, i.e., $i\to j$ means  $a_i\to b_j$.  

\section{Representing our shuffling using a poset structure}\label{sec:cycles}
The key to understanding this shuffling is looking at a column which will (for some $\ell\in\{0,\ldots,n-1\}$) consist of the terms
\[
\ell,\ell+n,\ell+2n,\ldots,\ell+(k-1)n,
\]
and after sorting will be sent to 
\[
k\ell,k\ell+1,k\ell+2,\ldots,k\ell+(k-1).
\]
For example, returning to the case $N=12$ and $k=3$ then we get the following.

\bigskip

\noindent\hfil\includegraphics[scale=0.9]{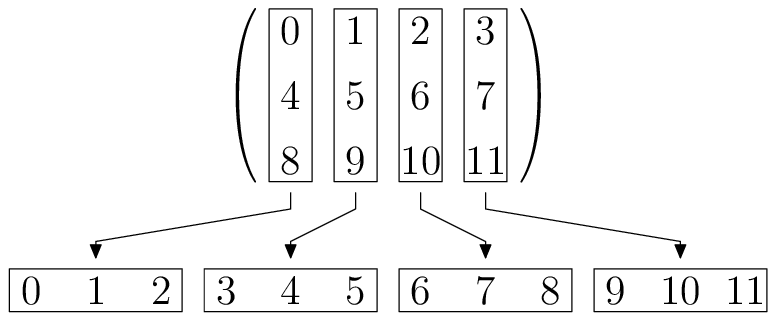}\hfil

\bigskip

\noindent This shows, for example,  $\{a_0,a_4,a_8\}\to\{b_0,b_1,b_2\}$ in some order (depending on the labels of the cards).  To help us understand what is happening it is useful to weight the subscripts.

\begin{definition}
A shuffling weight function on the subscripts is a map $\varphi:\{0,\ldots,N-1\}\to \mathbb{Z}$ which satisfies the following two conditions for $\ell\in\{0,1,\ldots,n-1\}$:
\begin{itemize}
\item[(i)] 
$\big\{\varphi(\ell),\varphi(\ell+n),\ldots,\varphi(\ell+(k-1)n)\big\}=
\big\{\varphi(k\ell),\varphi(k\ell+1),\ldots,\varphi(k\ell+(k-1))\big\}$.
\item[(ii)]  $\varphi(k\ell)<\varphi(k\ell+1)<\cdots<\varphi(k\ell+(k-1))$.
\end{itemize}
\end{definition} 

The first condition says  the weight of the entries in the column and the weight of the entries in a block it maps to are equal.  The second condition says  the weights are distinct and increasing in a block (and combined with the first condition says  the weights of the column are also distinct).

As an example for $N=12$ and $k=3$ one possible weight function is given below.
\[
\begin{array}{||c|c|c|c|c|c|c|c|c|c|c|c|c||} \hline\hline
n&0&1&2&3&4&5&6&7&8&9&10&11\\ \hline
\varphi(n)&0&1&2&0&1&2&0&1&2&0&1&2 \\ \hline\hline
\end{array}
\]
If we now take this weight function and replace all the entries in the above diagram with the corresponding weights then our diagram now becomes the following (for which it is easy to check  the two conditions are satisfied).

\bigskip

\noindent\hfil\includegraphics[scale=0.9]{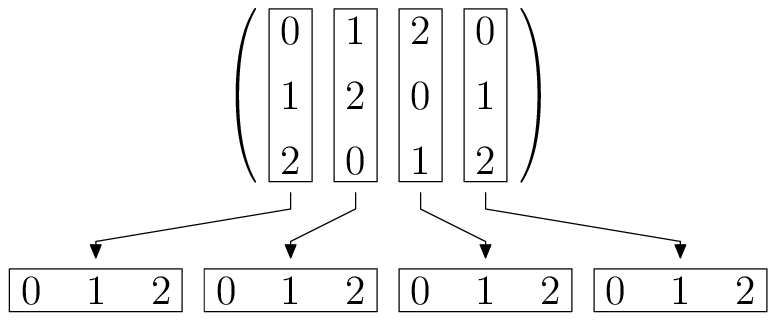}\hfil

\bigskip

We can use this weight function to map the subscripts to the subscripts in a bijective manner so  the weight is preserved.  For instance in the third column we have $\{2,6,10\}\to\{6,7,8\}$.  Since $\varphi(2)=2=\varphi(8)$, $\varphi(6)=0=\varphi(6)$ and $\varphi(10)=1=\varphi(7)$ then we would have $2\to 8$, $6\to 6$ and $10\to 7$.  Repeating this for each column/block combination we can now break up the elements of $\{0,1,\ldots,N-1\}$ into cycles, and we can place these cycles into a poset structure where the height of a cycle is the weight of a subscript in the cycle (by construction they are all equal).  So for our particular weight function we would end up with three levels with the following cycles.

\bigskip

\noindent\hfil\includegraphics[scale=0.9]{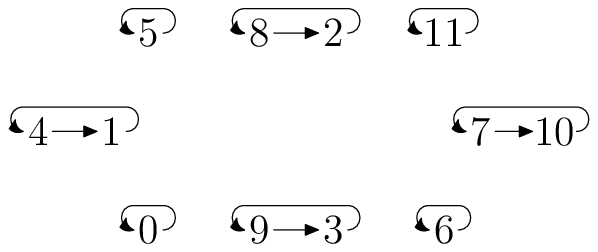}\hfil

\bigskip

This diagram essentially represents half of the shuffling (i.e., how we move from columns to blocks in the concatenation).  We also need to represent the other half of the shuffling, which is sorting among the columns.  To do this we add edges between levels in the diagram connecting any two subscripts which appear in the same column.  Doing this gives us the following diagram (which we will refer to as the {\em shuffling poset}\/).

\bigskip

\noindent\hfil\includegraphics[scale=0.9]{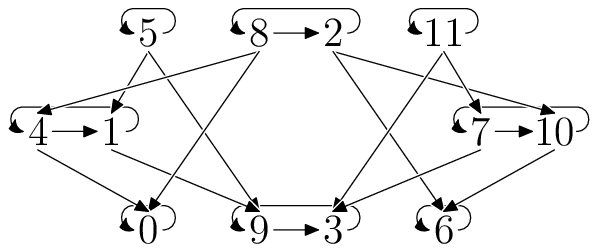}\hfil

\bigskip

The important part about the shuffling poset is that the higher ordered cards (or by analogy heaviest cards) will always try to ``sink down'', i.e., the highest label in a column will want to map to the smallest possible $\varphi$ value which corresponds to the lowest level it can reach, the second highest label will similarly go to the second lowest level it can reach and so on.  In particular, we have the following.

\begin{observation}
Shuffling can be carried out completely in the shuffling poset by placing the corresponding cards in their position and carrying out the following two steps:
\begin{itemize}
\item[(i)] Using only the edges {\em between}\/ levels in the poset swap cards so  no edge connects a higher ordered card at height (or weight) $i$ in the poset with a lower ordered card at height (or weight) $j$ where $i>j$.
\item[(ii)] Using only the edges {\em in}\/ the levels advance the card to the next entry.
\end{itemize}
\end{observation}

\subsection{Generating a shuffling weight function}\label{algo}
We now see that once we have a shuffling weight function we can generate the shuffling poset which completely describes the shuffling.  We now show  a shuffling weight function always exists by giving an algorithm which will generate one.
\begin{itemize}
\item[(1)] Set ${\tt weight}={\tt 0}$ and construct a ${\tt k}{\times}{\tt n}$ array where the element in the ${\tt (i,j)}$th entry of the array (${\tt 0}\leq {\tt i}\leq {\tt k-1}$ and ${\tt 0}\leq {\tt j}\leq {\tt n-1})$ is ${\tt a}:{\tt b}={\tt (i+kj)}:{\tt\big((i+kj)\%n)}$, where ${\tt s\%t}$ is the remainder when ${\tt s}$ is divided by ${\tt t}$.
\item[(2)] While there are still cells in the array which have not been crossed off do the following:
\begin{itemize}
\item[(i)] Construct a directed graph on the vertices ${\tt \{0,1,\ldots,n-1\}}$ by letting ${\tt j\to k}$ where ${\tt j}$ is the column and ${\tt k}$ is the ${\tt b}$ entry in the {\em lowest}\/ cell of column ${\tt j}$ which has not been crossed out.  (If all the cells in column ${\tt j}$ have been crossed out then ${\tt j}$ will be an isolated node.)
\item[(ii)] Find all directed cycles in the graph that was constructed.  For each edge in the directed cycle let $\varphi({\tt a})={\tt weight}$, where ${\tt a}$ comes from the cell that generated the edge.
\item[(iii)] Cross out all cells which were used to assign a weight.
\item[(iv)] Increase ${\tt weight}$ by ${\tt 1}$.
\end{itemize}
\end{itemize}

An example of the algorithm being carried out when $N=32$ and $k=4$ is shown in the Appendix.  (The reader is encouraged to try this algorithm out to generate the weight function for $N=12$ and $k=3$ given above and for $N=24$ and $k=6$ given in Section~\ref{sec:conclusion}.)

\begin{theorem}
The above algorithm generates a shuffling weight function.
\end{theorem}
\begin{proof}
First let us observe that the algorithm will terminate so  $\varphi$ is a well defined function.  This follows since each column is referred to exactly $k$ times and so if all of the entries in a column have been crossed out no edge will be directed towards it.  In particular, if we ignore isolated vertices then the directed graph generated in step 2i will have all vertices with outdegree $1$ and so it must contain at least one directed cycle (in fact one directed cycle for each connected component).  Therefore at every stage we will continue to cross out cells unless all cells have been crossed out so the algorithm will terminate.

The first condition of a shuffling weight function (that weights in the blocks match weights in the columns) is satisfied because we are pulling out cycles.  Namely, if we look at a cell ${\tt a}:{\tt b}$ in the ${\tt j}$th column, this can be understood as the subscript ${\tt a}$ will be placed into the ${\tt b}$th column (when shuffling).  What we need is to make sure  there is some subscript which will go into the ${\tt j}$th column that also will be assigned the same weight.  But this happens because we assign the entire directed cycle the same weight and by step 2ii the vertex preceding ${\tt j}$ will give a subscript that maps into the ${\tt b}$th column.

The second condition of a shuffling weight function (that weights in each block are increasing) is easily satisfied since the blocks form the columns of the array generated in step 1, and in columns the weights assigned to the cells are increasing  (since in each round at most one cell in each column will be assigned the current weight.)
\end{proof}

The algorithm used the lowest cell which had not been crossed out.  We could also have used the highest cell which was not crossed out and then decreased the weight by $1$.  This will essentially generate the same picture (where we rotated the arrays by $180$ and let ${\tt i}\to {\tt N-1-i}$ in the entries and in the directed graphs generated in step 2i of the algorithm.

Another important thing to note is  the directed cycles  we are pulling out are the same directed cycles (but with the ${\tt b}$ terms replaced by the ${\tt a}$ terms from the corresponding cells) found in the shuffling poset.  If we combine this with the previous idea of going from bottom to top we get the following observation.

\begin{observation}
In the shuffling poset  $i\to j$ if and only if $N-1-i\to N-1-j$.
\end{observation}

Something else worth noting is that if we look at the shuffling weight function produced in the Appendix we see there is some symmetry, i.e., $\varphi(i)+\varphi(31-i)$ is independent of $i$.  So for example we have $\varphi(2)+\varphi(29)=7=\varphi(12)+\varphi(19)$.  When this happens we will call this a {\em symmetrical}\/ weight function. 

\begin{conjecture}
The shuffling weight function produced by the algorithm is symmetric.
\end{conjecture}

It is easy to show  a symmetric weight function always exists for every $N$ and $k$.  For instance let $\varphi_{up}$ be the weight function generated by the given algorithm (where we use the lowest cells and work our way up) and let $\varphi_{down}$ be the weight function generated by the modified algorithm (where we use the highest cells and work our way down).  Then the weight function $\varphi=\varphi_{up}+\varphi_{down}$ is still a shuffling weight function and by the symmetry of the two algorithms will be symmetric.

\section{A simple weight function when $\gcd(q,k)=1$}\label{sec:gcd}
In the preceding section we saw that a weight function could be used to give us the shuffling poset, which in turn gives us another representation of how to shuffle.  One natural question to ask is what can we say about the lengths of the directed cycles in the shuffling poset.  In this section we will show in the special case when $N=k^tq$ and $\gcd(q,k)=1$  this question has an easy answer.   Note  this will cover all the cases when $k$ is a prime (in particular $k=2$).  First, for $\gcd(q,k)=1$ we observe  one shuffling weight function can be found using the base $k$ expansion of subscripts.

\begin{lemma}\label{lem:qk}
Let $N=kn=k^tq$ with $\gcd(q,k)=1$ and let $\ldots A_tA_{t-1}\ldots A_1A_0$ be the base $k$ expansion of $A$.  Then  $\varphi(A)=A_0+\cdots+A_{t-1}$ is a shuffling weight function.
\end{lemma}
\begin{proof}
Let $\ell\in\{0,1,\ldots,n-1\}$ and $\ldots L_tL_{t-1}\ldots L_1L_0$ be the base $k$ expansion of $\ell$, and let the base $k$ expansion of $n$ be $\ldots u0\ldots00$ where $u\neq 0$ is in the $(t-1)$th slot and $\gcd(u,k)=\gcd(q,k)=1$ (this follows from $n=k^{t-1}q$).

In particular, for $i\in\{0,\ldots,k-1\}$ the base $k$ expansion of $k\ell+i$ will be
\[
\ldots L_{t-1}L_{t-2}\ldots L_0i
\]
which has $\varphi(k\ell+i)=i+L_0+\cdots+L_{t-2}$.  On the other hand the base $k$ expansion of $\ell +in$ will be
\[
\ldots \big((L_{t-1}+iu)\%k\big)L_{t-2}\ldots L_1L_0
\]
which has $\varphi(\ell+iu)=\big((L_{t-1}+iu)\%k\big)+L_0+\cdots+L_{t-2}$ (where $s\%t$ is the remainder of $s$ when divided by $t$).  Because $\gcd(u,k)=1$ then $(L_{t-1}+iu)\%k$ will cycle through all $k$ possibilities as $i$ goes from $0$ to $k-1$.  It follows  the first condition of a shuffling weight function is satisfied.

The second condition of a shuffling weight function is satisfied since the base $k$ expansion of $k\ell+i$ for $i\in\{0,\ldots,k-1\}$ is $\ldots L_{t-1}L_{t-2}\ldots L_0i$ and so
\[
\varphi(k\ell+i)=i+L_0+\cdots+L_{t-2}.
\]
From this it follows  $\varphi(k\ell)<\varphi(k\ell+1)<\cdots<\varphi(k\ell+(k-1))$.
\end{proof}

One nice feature about this case is the the rule for mapping subscripts is easy to describe.  Namely, if the base $k$ expansion of $A$ is $\ldots A_{t-1}A_{t-2}\ldots A_1A_0$ then the map is
\begin{equation}\label{eq:map}
A\to kA+A_{t-1}\pmod{N}.
\end{equation}
(There are two things to check, one is that both these terms have the same weight and the other is that the column that contains $A$ will map to the block that contains $kA+A_{t-1}$, both conditions are easily checked.)

We will now use this map to determine the cycle lengths in the shuffling poset.

\begin{theorem}\label{thm:period}
Let $N=k^tq$ with $\gcd(k,q)=1$, and let $\order_k(s)$ denote the multiplicative order of $k$ modulo $s$.  Then the length of a cycle in the shuffling poset when we divide $N$ into $k$ equal stacks for shuffling is a divisor of $\order_k(N-q)$.  Further, there is a cycle of length $\order_k(N-q)$ in the shuffling poset.
\end{theorem}

Before we begin we note  by our assumption that $\gcd(N-k,q)=1$ and so the multiplicative order is well defined.  As a check we note  $\order_3(12-4)=\order_3(8)=2$ since $3^2=9\equiv 1\pmod{8}$, this agrees with the diagram given above for $N=12$ and $k=3$.

\begin{proof}
Consider a cycle  starting at $x$, and suppose that the base $k$ expansion of $x$ is $\ldots A_{t-1}\ldots A_1A_0$.  Using \eqref{eq:map} we have after $t$ steps we will be at
\begin{eqnarray*}
x&\rightarrow &kx+A_{t-1}\pmod{N}\\
&\rightarrow &k^{2}x+kA_{t-1}+A_{t-2}\pmod{N}\\
&\rightarrow &k^{3}x+k^{2}A_{t-1}+kA_{t-2}+A_{t-3}\pmod{N}\\
&\rightarrow &\cdots\\
&\rightarrow &k^{t}x+\underbrace{\sum_{i=0}^{t-1}k^{i}A_{i}}_{=A'}\pmod{N}.
\end{eqnarray*}
In particular, after we have taken $t$ steps, the last $t$ terms in the base $k$ expansion will agree with the last $t$ terms in the base $k$ expansion of $x$.  Now suppose  we repeat this $r$ times (so a total of $rt$ steps).  Then we have
\begin{eqnarray*}
x&\rightarrow &k^tx+A'\pmod{N}\\
&\rightarrow &k^{2t}x+k^tA'+A'\pmod{N}\\
&\rightarrow &k^{3t}x+k^{2t}A'+k^tA'+A'\pmod{N}\\
&\rightarrow &\cdots\\
&\rightarrow &k^{rt}x+\sum_{i=0}^{r-1}k^{it}A'\pmod{N}.
\end{eqnarray*}
For some $r$ we will be back where we started if
\[
k^{rt}x+\sum_{i=0}^{r-1}k^{it}A'\equiv x \pmod{N=k^tq}.
\]
Multiplying both sides by $k^t-1$ and simplifying this is equivalent to
\[
(k^{rt}-1)\big(x(k^t-1)+A'\big)\equiv 0\pmod{(k^t-1)k^tq}
\]
Looking at the base $k$ expansion of $x$ we have $x=A'+mk^t$ for some $m$, if we now substitute this in and simplify we get
\[
(k^{rt}-1)k^t\big(A'+m(k^t-1)\big)\equiv 0\pmod{(k^t-1)k^tq}
\]
or
\begin{equation}\label{eq:mod}
(k^{rt}-1)\big(A'+m(k^t-1)\big)\equiv 0\pmod{(k^t-1)q}.
\end{equation}
Since $(k^t-1)q=N-q$, if $rt=\order_k(N-q)$ then $k^{rt}-1\equiv 0\pmod{(k^t-1)q}$ and this condition is satisfied.  In particular, after taking $\order_k(N-q)$ steps then $x\rightarrow  x$ for each value of $x$, and so each cycle must be some divisor of $\order_k(N-q)$.

We are implicitly using  $t\,\big|\, \order_k(N-q)$.  To see why this is true, let $M=\order_k(N-q)$ then $N-q=(k^t-1)q\,\big|\,(k^M-1)$, and so
\[
(k^t-1)\,\big|\,(k^M-1)\mbox{~~so~~}\gcd(k^t-1,k^M-1)=k^t-1,
\]
but we also have  in general
\[
\gcd(k^a-1,k^b-1)=k^{\gcd(a,b)}-1\mbox{~~so~~}\gcd(k^t-1,k^m-1)=k^{\gcd(t,m)}-1.
\]
Combining these two statements we have  $\gcd(t,M)=t$ showing  $t$ is a divisor of $M=\order_k(N-q)$.

Finally, to show  there is a cycle of length $\order_k(N-q)$ we note  \eqref{eq:mod} also holds for $x=1$, which corresponds to $A'=1$ and $m=0$.  So  \eqref{eq:mod} reduces to finding the smallest value of $rt$ so 
\[
k^{rt}\equiv 1\pmod{N-q},
\]
which is clearly $rt=\order_k(N-q)$.
\end{proof}

\subsection{A more generalized weight function}
The preceding weight function relied on having $\gcd(q,k)=1$ (this was used in the lemma to make sure  we hit all of the residue classes modulo $k$, and then in the Theorem~\ref{thm:period} by giving a simple rule for mapping).  We now present a weight function that works in more cases (but from the definition it will agree with the previous weight function when $\gcd(q,k)=1$).  For example, the following weight function will work for the case when $k$ is any prime power.

\begin{lemma}
Let $N=kn=k^tq$ with $\gcd\big(q/\gcd(q,k),\gcd(q,k)\big)=1$ and let $\ldots A_tA_{t-1}A_{t-2}\ldots A_1A_0$ be the base $k$ expansion of $A$.  Then
$\varphi(A)=A_0+\cdots+A_{t-1}+\big(A_t\%\gcd(k,q)\big)$
is a shuffling weight function.
\end{lemma}
\begin{proof}
For $\ell\in\{0,\ldots,n-1\}$ let $\ldots L_{t}L_{t-1}\ldots L_1L_0$ be the base $k$ expansion of $\ell$.  Then for $i\in\{0,\ldots,k-1\}$, the base $k$ expansion of $k\ell+i$ is $\ldots  L_{t-1}L_{t-2}\ldots L_1L_0i$ and so
\[
\varphi(k\ell+i)=i+L_0+\cdots+L_{t-2}+\big(L_{t-1}\%\gcd(k,q)\big).
\]
Using this, we note  the second condition for a shuffling weight function is easily satisfied.

Since $n=k^{t-1}q$, the base $k$ expansion of $n$ is $\ldots N_tN_{t-1}0\ldots0$.  Also we have  $N_{t-1}=(q\%k)$, and so $\gcd(q,k)=\gcd\big((q\%k),k)=\gcd(N_{t-1},k)$.  So for $i\in\{0,\ldots,k-1\}$ the base $k$ expansion of $\ell+in$ is $\ldots C_tC_{t-1}L_{t-2}\ldots L_0$, i.e., it agrees in the first $t-2$ slots with the expansion of $\ell$, so we need to understand $C_{t-1}+\big(C_t\%\gcd(q,k)\big)$.

For a given value of $i$ we have $C_{t-1}=\big((L_{t-1}+iN_{t-1})\%k\big)$.  Since $N_{t-1}/\gcd(q,k)$ is relatively prime to $k$,  this takes  on $k/\gcd(q,k)$ different values that differ by a multiple of $\gcd(q,k)$ from $A_{t-1}$ (modulo $k$), and we will attain each one of these values $\gcd(q,k)$ times as $u$ ranges over its possible values.  In particular we have,
\begin{multline*}
C_{t-1}\in\bigg\{\big(L_{t-1}\%\gcd(q,k)\big),\big(L_{t-1}\%\gcd(q,k)\big)+\gcd(q,k),\\
\ldots,\big(L_{t-1}\%\gcd(q,k)\big)+\bigg({k\over \gcd(q,k)}-1\bigg)\gcd(q,k)\bigg\}
\end{multline*}
Fix a possible value for $C_{t-1}$ and let $\widehat{i}$ be the first value of $i$ that gives us $C_{t-1}$, and let $\widehat{C}_t$ be the value of $C_t$ for this corresponding $i$.  In particular, the only values of $i$ where we will be at the fixed value of $C_{t-1}$ are $i=\widehat{i}+mk/\gcd(q,k)$ for $m\in\{0,\ldots,\gcd(q,k)-1\}$.  For these values of $u$ we have 
\[
C_{t}=\bigg(\bigg(\widehat{C}_t+m{(N_{t-1}+kN_t)\over \gcd(q,k)}\bigg)\%k\bigg)=\bigg(\bigg(\widehat{C}_t+m{q\over \gcd(q,k)}\bigg)\%k\bigg).
\]
(The last step follows from noting $q=N_{t-1}+kN_t+k^2N_{t+1}+\cdots$ and that when we divide this by $\gcd(q,k)$ all but the first two terms will have at least one factor of $k$.)  So we have
\[
\big(C_t\%\gcd(q,k)\big)=\bigg(\bigg(\widehat{C}_t+m{q\over \gcd(q,k)}\bigg)\%\gcd(q,k)\bigg).
\]
Since $\gcd\big(q/\gcd(q,k),\gcd(q,k)\big)=1$ then this covers all of the residue classes modulo $\gcd(q,k)$.  Combined with what we know about $C_{t-1}$ then we have
\begin{multline*}
C_{t-1}+\big(C_t\%\gcd(q,k)\big)=\\ \big\{\big(L_{t-1}\%\gcd(q,k)\big),\big(L_{t-1}\%\gcd(q,k)\big)+1,\ldots,\big(L_{t-1}\%\gcd(q,k)\big)+k-1\big\}.
\end{multline*}
So we have 
\[
\varphi(\ell+in)=j+L_0+\cdots+L_{t-2}+\big(L_{t-1}\%\gcd(k,q)\big),
\]
where $j$ which ranges over $\{0,\ldots,k-1\}$ as $i$ ranges over $\{0,\ldots,k-1\}$.
\end{proof}

The problem is  even though we have a simple weight function, the rule for mapping is not simple, and so we have no similar result as Theorem~\ref{thm:period}.  We note one of the results coming out of the proof of Theorem~\ref{thm:period} is  when $\gcd(q,k)=1$ the cycle in the shuffle poset which contains $1$ has maximal length.  This no longer needs to hold when $\gcd(q,k)\neq 1$.  If we look at the example shown in the Appendix, we can see  the cycle which will contain $1$ has length $3$ but the maximal length of a cycle is $6$.

\section{Finding fixed and periodic stacks}\label{sec:posets}
In Section~\ref{sec:cycles} we saw a way to represent our shuffling in terms of a poset of cycles.  We will now exploit this poset to find stacks of cards which are fixed or periodic under shuffling.

To find the fixed stacks, we observe  when looking at the cycles of the shuffling poset in a stack which is fixed under shuffling there are two necessary conditions:
\begin{itemize}
\item[(i)] If there is an edge between levels in the shuffling poset then the label on the lower level must be at least as great as the label on the higher level.  (Otherwise we would swap labels and we would not be fixed.)
\item[(ii)] All the labels in a cycle on a level in the shuffling poset must agree.  (Otherwise when we shift by one in the cycle the stack is not fixed.)
\end{itemize}
It is easy to see that these conditions are also sufficient.

We now form a poset, which we call the {\em fixed poset}, based off the shuffling poset.  Namely, each cycle goes to one element in the poset (at the same height as before) and we connect an edge between two corresponding cycles if one cycle contains $c$ while the other contains $d$ from the same column and there is no $e$ in the column with $\varphi(c)<\varphi(e)<\varphi(d)$.  (Technically, by condition (i) above we would want to connect all cycles which are connected by an edge in the shuffle poset, but by transitivity we only need to consider edges which cannot be broken down further.)  Some examples of fixed posets are given in Figure~\ref{fig:fixpo}.

\begin{figure}[hft]
\centering
\subfloat[$\displaystyle{N=8\atop k=2}$]{\includegraphics{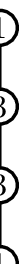}}\hfil
\subfloat[$\displaystyle{N=8\atop k=4}$]{\includegraphics{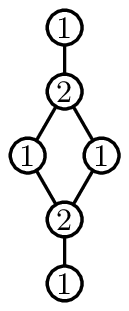}}\hfil
\subfloat[$\displaystyle{N=12\atop k=2}$]{\includegraphics{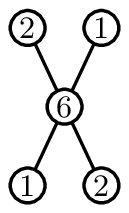}}\hfil
\subfloat[$\displaystyle{N=12\atop k=3}$]{\includegraphics{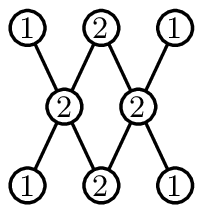}}\hfil

\subfloat[$\displaystyle{N=24\atop k=2}$]{\includegraphics{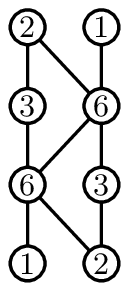}}\hfil
\subfloat[$\displaystyle{N=24\atop k=3}$]{\includegraphics{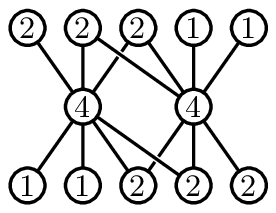}}\hfil
\subfloat[$\displaystyle{N=24\atop k=4}$]{\includegraphics{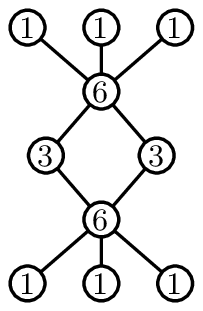}}\hfil
\subfloat[$\displaystyle{N=52\atop k=2}$]{\includegraphics{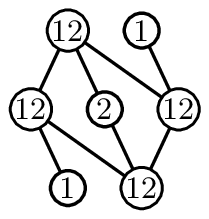}}\hfil
\subfloat[$\displaystyle{N=52\atop k=4}$]{\includegraphics{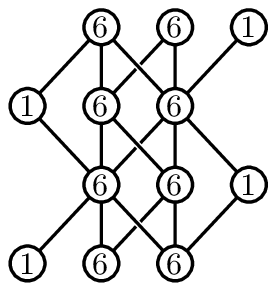}}\hfil
\caption{Fixed posets for various values of $N$ and $k$.}
\label{fig:fixpo}
\end{figure}

Using the two conditions, the fixed stacks are now found by assigning labeled cards to each element in the fixed poset so  the labels are weakly increasing with respect to the fixed poset.  This allows us to quickly and easily find fixed stacks for a given $N$ and $k$.  

It also helps us to answer whether or not fixed stacks with some given property can exist.  For example if $N=52$ and $k=2$ it is easy to see from the fixed poset there is no fixed stack with four labels each label with 13 cards (i.e., such as in a standard deck of cards with the labels being the suits in some order).

Note  the fixed posets in Figure~\ref{fig:fixpo} are symmetric in that if we flipped it upside down we would have the same poset.  This will always happen when our weight function is symmetric (and as already noted we can always find such a weight function).

One question to consider is the number of fixed stacks.  This is dependent on both the fixed poset and the number of labels.  As an example, for a fixed $k$ and a fixed order of the labeling  we can have dramatically different behavior for various values of $N$.  For example, suppose  we have two labels $1\succ 0$ and $k=2$.  If $N=2^t$, then by Theorem~\ref{thm:period} all the cycles have length at most $t$ and there are $t+1$ levels in the fixed poset using the weight function given in Lemma~\ref{lem:qk}.  We can then assign all of the cycles with height $<t/2$ in the poset to have label $1$ and then take any collection of posets from the layer at height $\lceil t/2\rceil$ to have label $1$.  There are at least
\[
{1\over t}{t\choose \lceil t/2\rceil}\approx {cN\over(\ln N)^{3/2}},
\] 
cycles on the $\lceil t/2\rceil$ level for some constant $c>0$, and since we can take any subset of them and form a fixed stack then there are at least $2^{cN/(\ln N)^{3/2}}$ fixed stacks for $N=2^t$ (in particular this is super-polynomial).  By comparison if we consider $N=4{\cdot}3^t$, then it is easy to show  it has a fixed poset of the following form.

\bigskip

\noindent\hfil\includegraphics[scale=0.9]{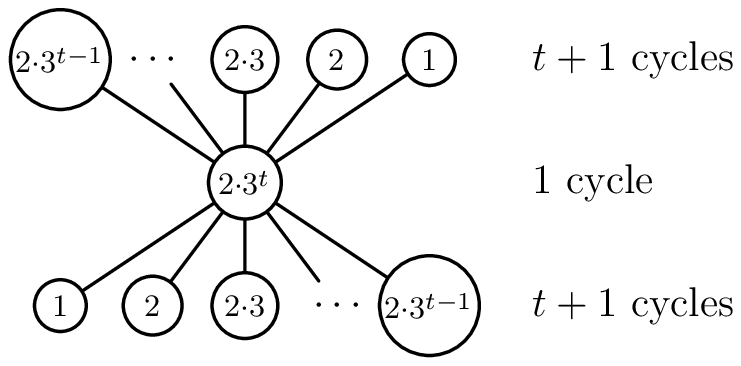}\hfil

\bigskip

So in particular there are exactly $4{\cdot}2^{t}\leq cN^{0.405}=o(N)$ fixed stacks for $N=4{\cdot}3^t$.

To make the comparison more concrete we have  for $N=1024=2^{10}$ there are $292,592,830$ fixed stacks, while for $N=972=4{\cdot}3^5$ there are 128 fixed stacks.

\subsection{Finding periodic stacks}
When looking for periodic stacks we will again set up a poset, but instead of having cycles as elements in the poset we will have individual subscripts be the elements.  The basic idea is to consider what can keep a card at a given subscript from dropping down as we go through repeated shufflings  (because we are periodic the only operation that will happen as we shuffle is shifting the cycles).  In particular, an element $A$ can only drop down if for some $B$ and some $s\geq0$ we have the following edges in our shuffling poset:
\[
\begin{array}{c@{}c@{}c@{}c@{}c@{}c@{}c@{}c@{}c}
A=c_0&\rightarrow&c_1&\rightarrow&\cdots&\rightarrow&c_{s-1}&\rightarrow&c_s\\
&&&&&&&&\big\downarrow\\
B=d_0&\rightarrow&d_1&\rightarrow&\cdots&\rightarrow&d_{s-1}&\rightarrow&d_s
\end{array}
\]

In this case we need to make sure  the card in $B$ will not cause the card in $A$ to sink.  To do this we draw the poset where the elements are the subscripts (as before we can place $A$ at level $\varphi(A)$), and we connect an edge between $A$ and $B$ if we have the edges in our shuffling poset as indicated above {\em and}\/ there is no $e$ in the same column as $c_s$ and $d_s$ so  $\varphi(c_s)<\varphi(e)<\varphi(d_s)$.  We will call this the {\em periodic poset}\/.  An example of the situation is shown in Figure~\ref{fig:1624}.

\begin{figure}[ht]
\centering
\subfloat[Shuffling poset]{\includegraphics[scale=0.8]{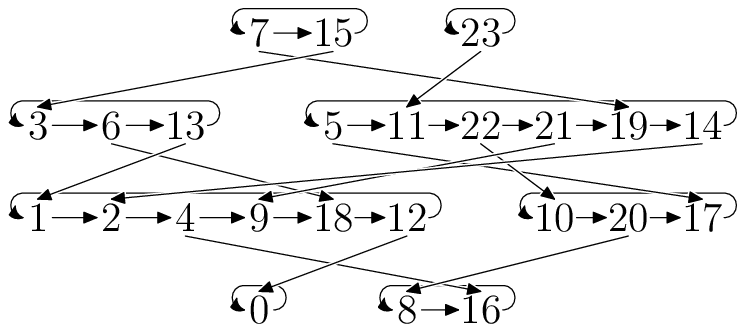}}\hfil
\subfloat[Periodic poset]{\includegraphics[scale=0.7]{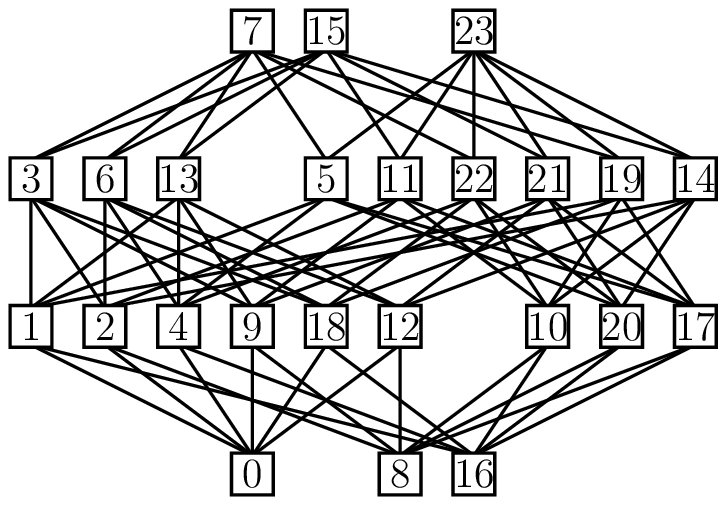}}
\caption{Posets for the case $N=24$ and $k=2$.}
\label{fig:1624}
\end{figure}

As before,  the periodic stacks are now found by assigning labeled cards to each subscript in the periodic poset so  the labels are weakly increasing with respect to the periodic poset.  This allows us to quickly and easily find periodic stacks for given $N$ and $k$.  

The possible periods of the periodic stacks are determined by the size of the cycles in the shuffling poset.  Namely, the possible periods are the divisors of the least common multiple of the cycle lengths in the shuffling poset.  In the case when $\gcd(q,k)=1$ then Theorem~\ref{thm:period} shows  the possible periods of stacks are divisors of $\order_k(N-q)$, on the other hand it is easy to construct a periodic stack for any period dividing $\order_k(N-q)$.

\section{Concluding remarks}\label{sec:conclusion}
We have seen how to find a weight function which can in turn allow us to represent our shuffling in a shuffling poset.  This poset can then be modified to help us find fixed stacks, periodic stacks, and also tell us which periods are possible.

One of the problems  we have not addressed is how quickly a stack will settle into a periodic stack.  As with the number of fixed stacks this depends highly on $N$.  For example, it is not hard to see (i.e., using the periodic poset) that it takes no more than $3tm$ shuffles to settle into a periodic orbit (where $m$ is the least common multiple of the cycle lengths and $t+1$ the number of levels in the shuffling poset).  So for example when $N=2^t$ and $k=2$ then we need at most $3\big(\ln N\big)^2$ steps.  On the other hand for $N=4{\cdot}3^k$ it is easy to construct a stack that takes exactly $N/2$ steps to get into a periodic stack.

There are still many questions that remain.  One of the biggest problems is trying to understand how the weight function works for arbitrary $N$ and $k$.  For instance the weight functions given in Section~\ref{sec:gcd} do not apply for the case $N=24$ and $k=6$.  In this case the algorithm for finding a weight function generates the following:
\[
\begin{array}{l}
\begin{array}{||c||c|c|c|c|c|c|c|c|c|c|c|c||} \hline\hline
n&0&1&2&3&4&5&6&7&8&9&10&11\\ \hline
\varphi(n)&0&1&4&5&6&7&1&2&5&6&7&8 \\ \hline\hline
\end{array}\\[20pt]
\begin{array}{||c||c|c|c|c|c|c|c|c|c|c|c|c||} \hline\hline
n&12&13&14&15&16&17&18&19&20&21&22&23\\ \hline
\varphi(n)&1&2&3&4&7&8&2&3&4&5&8&9 \\ \hline\hline
\end{array}
\end{array}
\]

One notable difference between this weight function and the weight function when $\gcd(q,k)=1$ or $\gcd\big(q/\gcd(q,k),\gcd(q,k)\big)=1$ is that in the blocks the weight function does not consist of consecutive numbers, i.e., it has gaps.  Determining why these gaps are there and where they will appear given $N$ and $k$ will go a long way to understanding the shuffling weight function.

Another important question in regards to the shuffling posets is understanding the possible cycle lengths.  We understand what is going on for the case when $\gcd(q,k)=1$, but all other cases remain open.  For example, is it true  the least common multiple of the cycle lengths is the length of the longest cycle?

We can also consider what happens when instead of only considering a single type of shuffling we consider combining the $j!$ different shuffling rules that come from all the possible rearrangements of the ordering of the labels.  And of course, perhaps the most important thing missing right now is a good magic trick that can be performed using this shuffling rule, which was the original motivation of Larry Carter and J.-C.\ Reyes who first suggested this problem!

\section*{Appendix}
We implement the algorithm given in Section~\ref{algo} to find the weight function in the case $N=32$ and $k=4$.  The steps are shown below.

\bigskip

\noindent\hfil\includegraphics[scale=0.75]{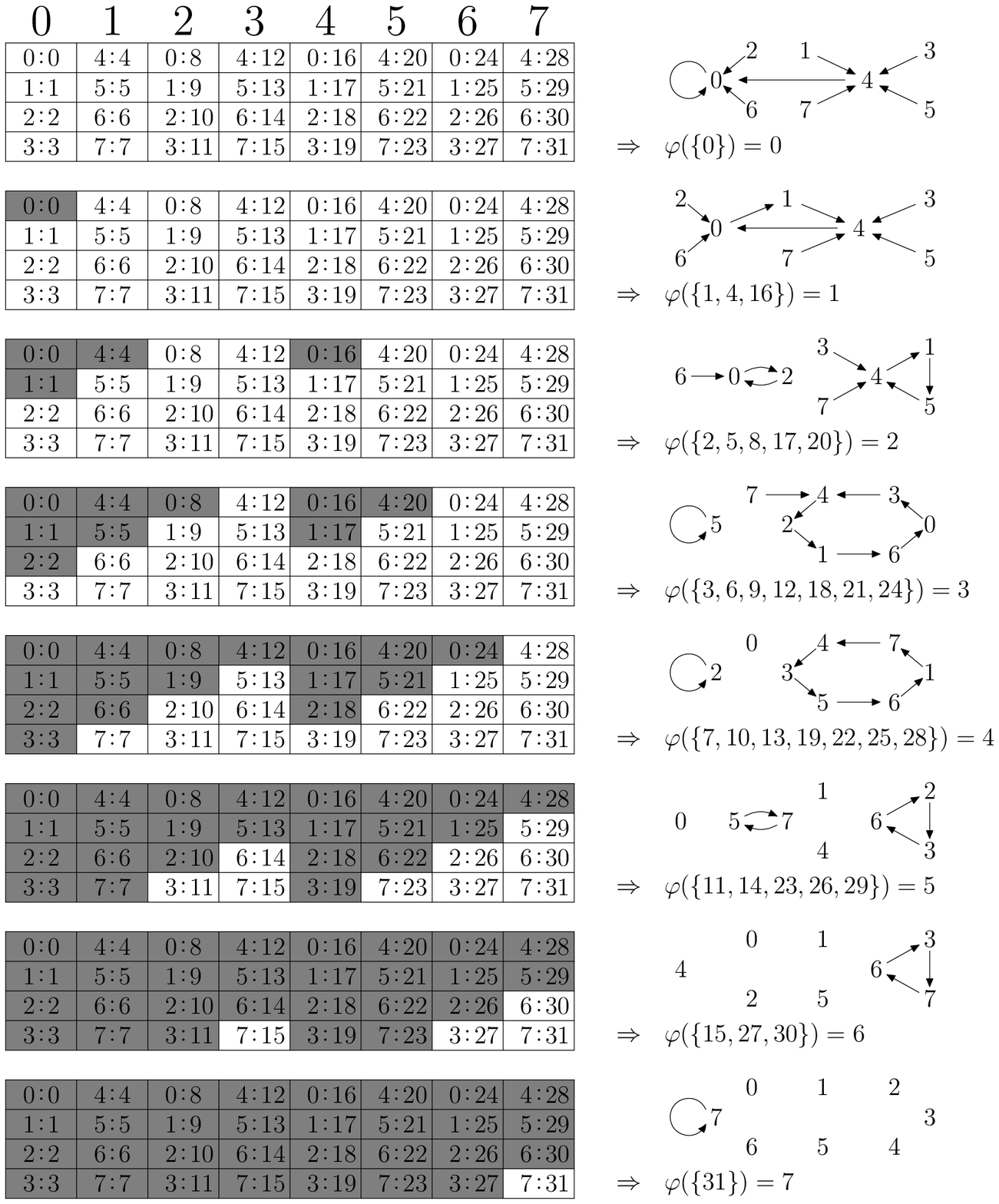}\hfil

\bigskip

The generated weight function is given in the following table.
\[
\begin{array}{l}
\begin{array}{||c||c|c|c|c|c|c|c|c|c|c|c|c|c|c|c|c|c||} \hline\hline
n&0&1&2&3&4&5&6&7&8&9&10&11&12&13&14&15&16\\ \hline
\varphi(n)&0&1&2&3&1&2&3&4&2&3&4&5&3&4&5&6&1 \\ \hline\hline
\end{array}\\[20pt]
\begin{array}{||c||c|c|c|c|c|c|c|c|c|c|c|c|c|c|c||} \hline\hline
n&17&18&19&20&21&22&23&24&25&26&27&28&29&30&31\\ \hline
\varphi(n)&2&3&4&2&3&4&5&3&4&5&6&4&5&6&7 \\ \hline\hline
\end{array}
\end{array}
\]


\begin{thebibliography}{99}
\bibitem{perfect}
Persi Diaconis, Ron Graham and William Kantor, \emph{The mathematics of perfect shuffles}, Adv. in Appl. Math. {\bf 4} (1983), 175--196.
\bibitem{morris}
S. Brent Morris, \emph{Magic tricks, card shuffling and dynamic computer memories}, MAA Spectrum, Mathematical Association of America,
Washington, D.C., (1998) xviii + 148 pp.
\end{thebibliography}
\end{document}